\documentclass[12pt]{amsart}
\usepackage{
	amsmath,  amssymb,  amsthm,   amscd,
	gensymb,  graphicx, etoolbox, booktabs,
	stackrel, mathtools    
}
\usepackage[usenames,dvipsnames]{xcolor}
\definecolor{darkblue}{rgb}{0,0,0.7}
\definecolor{darkred}{rgb}{0.7,0,0}
\usepackage[colorlinks=true, linkcolor=darkred, citecolor=darkblue, urlcolor=blue, pagebackref=true, breaklinks=true]{hyperref}
\usepackage[capitalise]{cleveref}
\usepackage{enumitem}
\usepackage{placeins}
\usepackage{relsize}
\usepackage{todonotes}
\usepackage{soul}
\usepackage{tikz}
\usetikzlibrary{arrows}
\usetikzlibrary{shapes}
\usepackage{float}
\usepackage{tabularx}
\usepackage{kantlipsum}
\usepackage{array}
\usepackage[margin=2.5cm]{geometry}

\newtheorem{proposition}{Proposition}[section]
\newtheorem{lemma}[proposition]{Lemma}
\newtheorem{theorem}[proposition]{Theorem}
\newtheorem{corollary}[proposition]{Corollary}
\newtheorem{question}[proposition]{Question}

\newtheorem{remark}[proposition]{Remark}

\newtheorem{example}[proposition]{Example}

\newtheorem{innercustomthm}{Theorem}
\newenvironment{customthm}[1]
  {\renewcommand\theinnercustomthm{#1}\innercustomthm\itshape}
  {\endinnercustomthm}

\newtheorem{innercustomcor}{Corollary}
\newenvironment{customcor}[1]
  {\renewcommand\theinnercustomcor{#1}\innercustomcor\itshape}
  {\endinnercustomcor}
  
\newcommand{\reg}{{\rm reg}}
\newcommand{\G}{{\mathcal{G}}}

\def\H{\mathcal{H}}

\tikzstyle{place}=[draw,circle,minimum size=1mm,inner sep=1pt,outer sep=-1.1pt,fill=black]

\usetikzlibrary{shapes.geometric}
\tikzstyle{places}=[draw,rectangle,minimum size=8pt,inner sep=0pt]
\tikzstyle{placesf}=[draw,rectangle,minimum size=5pt,inner sep=0pt]
\tikzstyle{placec}=[draw,circle,minimum size=8pt,inner sep=0pt]
\tikzstyle{placecf}=[draw,circle, minimum size=5pt,inner sep=0pt]

\def\K{\mathbb{K}}
\def\reg{\mathrm{reg}}

\def\lcm{\mathrm{lcm}}

\def\L{\mathrm{lcm}}
\def\pd{\mathrm{pd}}

\def\H{\mathcal H}

\def\E{\mathcal E}

\def\G{\mathcal G}
\def\A{\mathcal A}
\def\B{\mathcal B}
\def\U{\mathcal U}
\def\x{\mathbf x}

\def\l{\langle}
\def\r{\rangle}
\def\b{\mathrm{bight}}

\def\p{\mathfrak p}
\def\q{\mathfrak q}
\def\h{\mathrm{ht}}

\begin{document}

\title[Stanley-Reisner ideals of higher independence complexes of chordal graphs]{Stanley-Reisner ideals of higher independence complexes of chordal graphs}

\author{Kanoy Kumar Das}
\address{Chennai Mathematical Institute, India}
\email{kanoydas0296@gmail.com; kanoydas@cmi.ac.in}

\author{Amit Roy}
\address{Chennai Mathematical Institute, India}
\email{amitiisermohali493@gmail.com}

\author{Kamalesh Saha}
\address{Department of Mathematics, SRM University-AP, Amaravati 522240, Andhra Pradesh, India}
\email{kamalesh.s@srmap.edu.in; kamalesh.saha44@gmail.com}

\keywords{connected ideals, chordal graphs, regularity, projective dimension, linear resolution, Cohen-Macaulay property}
\subjclass[2020]{05E40, 13F55, 05C05}

\vspace*{-0.4cm}
\begin{abstract}
    For $t\geq 2$, the $t$-independence complex $\mathrm{Ind}_t(G)$ of a graph $G$ is the collection of all $A\subseteq V(G)$ such that each connected component of the induced subgraph $G[A]$ has at most $t-1$ vertices. The topology of $\mathrm{Ind}_t(G)$ is intimately related to the combinatorial property of $G$. In this article, we consider the Stanley-Reisner ideal $J_{t}(G)$ of $\mathrm{Ind}_t(G)$ and focus on its algebraic properties. We prove that for a chordal graph $G$ and for all $t$
    \[
        \mathrm{reg}(R/J_{t}(G))=(t-1)\nu_{t}(G) \text{ and } \mathrm{pd}(R/J_{t}(G))=\mathrm{bight}(J_{t}(G)),
    \]
    where $\nu_{t}(G)$ denotes the induced matching number of the corresponding hypergraph of $J_{t}(G)$, and $\mathrm{reg}$, $\mathrm{pd}$ and $\mathrm{bight}$ stand for the regularity, projective dimension, and big height, respectively. As a consequence of the above results, we combinatorially characterize when the Stanley-Reisner ideal of the $t$-independence complex of a chordal graph has a linear resolution as well as when it satisfies the Cohen-Macaulay property. The above formulas and their consequences can be seen as a nice generalization of the classical results corresponding to the edge ideals of chordal graphs.
    
\end{abstract}

\maketitle

\section{Introduction}

An integral focus of research in the area of commutative algebra is the study of monomial ideals, particularly square-free monomial ideals, due to their strong connections with combinatorics and topology. A general objective in this area is to express or translate algebraic properties of a certain class of ideals in terms of the combinatorial or topological properties of the associated objects. There are several ways to associate a combinatorial object with a square-free monomial ideal; the most popular among them are the following two: (a) using the Stanley-Reisner correspondence to associate an abstract simplicial complex, and (b) associating a simple hypergraph (or clutter). Both these identifications have certain advantages and are frequently used to study square-free monomial ideals. Additionally, many algebraic properties of an ideal in a polynomial ring depend on the characteristics of the base field. However, if these come from the structure of the associated combinatorial object, then they are independent of the choice of the base field.
\par

The graded minimal free resolution of a graded module gives insight into its structure and measures its complexity. Determining the minimal free resolution of a graded module is a computationally challenging task. So, researchers try to get some estimation of the minimal free resolution via two important invariants: (i) the Castelnuovo-Mumford regularity (or simply, regularity) that measures the width of a minimal free resolution, and (ii) the projective dimension, which gives the length of a minimal free resolution. These two invariants have been extensively investigated for several classes of monomial ideals, more notably in the case of edge ideals of graphs. Indeed, a celebrated theorem of Fr\"{o}berg \cite{Froberg1990} gave an algebraic interpretation of chordal graphs in terms of the linearity of the minimal free resolution of edge ideals. More precisely, the edge ideal $I(G)$ of a graph $G$ has a linear resolution if and only if the complement of $G$ is chordal. On the other hand, when $G$ is chordal, a precise combinatorial formula for the regularity and projective dimension of $I(G)$ is well-known. \par 

To extend the study to square-free monomial ideals, various generalizations of edge ideals, such as path ideals, clique ideals, etc., have been introduced (see \cite{ConcaDeNegri1998, Moradi2018}). Note that the edge ideal of a graph $G$ can be realized as the Stanley-Reisner ideal of a simplicial complex obtained from $G$, often referred to as the independence complex of $G$. In the literature, there is a notion of higher independence complexes of a graph $G$, generalizing the independence complex of $G$. For $t\geq 2$, the {\it $t$-independence complex} of a graph $G$, denoted by $\mathrm{Ind}_{t}(G)$, is the collection of all $A\subseteq V(G)$ such that each connected component of $G[A]$ has at most $t-1$ vertices. The members of $\mathrm{Ind}_t(G)$ are called $t$-independent sets of $G$. Note that $\mathrm{Ind}_2(G)$ is nothing but the well-known {\it independence complex} of $G$. The topological study of the higher independence complexes of graphs turns out to be a fruitful area of research and often produces very interesting results when seen from the perspective of the independence complexes. To motivate the reader, we briefly mention some of these results below.\par 

In 2003, Meshulam \cite{Meshulam2003} established some beautiful relations between the domination numbers of $G$ and the homology of $\mathrm{Ind}_2(G)$, which settled a Hall-type conjecture of Aharoni in the affirmative. Recently, Deshpande-Shukla-Singh \cite{DSS2022} extended Meshulam's result by relating the homology groups of $\mathrm{Ind}_{t}(G)$ with the distance $t$-domination number of $G$. It is important to note that the distance $t$-domination number is a well-known invariant in graph theory (see \cite{TianXu2009} and the references therein). In \cite{DSS2022}, the authors have also shown that $\mathrm{Ind}_t(G)$ of a chordal graph $G$ is either contractible or homotopy equivalent to a wedge of spheres. Note that for $t=2$, $\mathrm{Ind}_2(G)$ of a chordal graph $G$ is sequentially Cohen-Macaulay, whereas for each $t\ge 3$, the complexes $\mathrm{Ind}_t(G)$ may not be sequentially Cohen-Macaulay (see, for instance, \cite[Proposition 4.3]{ADGRS}), which is quite surprising. In \cite{PS2018}, Paolini and Salvetti established a connection between the twisted cohomology of the classical braid groups and the cohomology of higher independence complexes related to the corresponding Coxeter graphs (see also \cite{Salvetti2015}). The complexes $\mathrm{Ind}_t(G)$ also appeared in some purely combinatorial contexts, for instance:
\begin{enumerate}
    \item[$\bullet$] They appeared in the work of Szab\'o and Tardos \cite{Szabo2003}, where the authors introduced and discussed generalizations of the problem of independent transversal in graphs.

    \item[$\bullet$] The notion of $t$-independent set has been explored from a purely graph-theoretic point of view. Specifically, it is related to the idea of clustered graph coloring (see \cite{Sampathkumar1993, Wood2018}). 
\end{enumerate}

Recently, the Stanley-Reisner ideals of these complexes have been considered in \cite{ADGRS, AJM2024, DRSV23}. In this article, we broaden this study by considering the class of chordal graphs. It turns out that the monomial generators of the Stanley-Reisner ideal of $\mathrm{Ind}_t(G)$ correspond to the connected subgraphs of size $t$ in $G$, and because of this, the ideal $J_t(G)$ is also called the `$t$-connected ideal' in \cite{AJM2024}. The ideals $J_t(G)$ are a natural generalization of edge ideals as $J_{2}(G)=I(G)$. In this regard, one should note that the $t$-path ideals of graphs are also a generalisation of edge ideals, and for $t\leq 3$, the $t$-path ideals coincide with $J_t(G)$.

It is well-known that any square-free monomial ideal can be seen as an edge ideal of a simple hypergraph. Let $\H$ be a $t$-uniform hypergraph, and $I(\H)$ denote its edge ideal in a polynomial ring $R$. Then the regularity (respectively, projective dimension) of $R/I(\H)$ is bounded below by $(t-1)\nu(\H)$ (respectively, $\b(I(\H))$), where $\nu(\H)$ denote the induced matching number of $\H$. These bounds are attained for various classes of simple graphs (i.e., $2$-uniform hypergraphs), including for the chordal graphs (see \cite[Corollary 6.9]{HaVanTuyl2008} and \cite[cf. Theorem 3.2]{FranciscoHa2007}). In the setting of general $t$-uniform hypergraphs, only a few classes are currently known for which $\operatorname{reg}(R/I(\mathcal H))$ achieves the lower bound. For example, this is established for $t$-connected ideals of $t$-gap-free chordal graphs \cite{AJM2024}, for $t$-connected ideals of gap-free and claw-free graphs \cite{AJM2024}, and for $t$-path ideals of gap-free and claw-free graphs when $t=3,4,5,$ or $6$ (see \cite{Banerjee2017}). The bound is further achieved for all $t$-path ideals of gap-free, claw-free, and whiskered $K_{4}$-free graphs \cite{Banerjee2017}, for $3$-path ideals of chordal graphs (see \cite{2024pathideals}), and for $t$-connected ideals of co-chordal graphs \cite{DRSV23}, among other cases.
  \par

Let $\H(G,t)$ be a $t$-uniform hypergraph induced from a graph $G$ such that $\H(G,2)=G$. One of the natural questions in this context is to ask for which classes of such hypergraphs $\H(G,t)$ the well-known results corresponding to $I(G)$ carry forward to higher $t$. Note that among different classes of simple graphs, chordal graphs have garnered special attention due to their connections with various branches of mathematics and computer science, as well as the fact that several algebraic invariants of their edge ideals can be expressed in terms of combinatorial invariants of the underlying graphs. Thus, it is worthwhile to first explore the above question in the context of chordal graphs.\par 

In this paper, we investigate the Stanley-Reisner ideals $J_t(G)$ corresponding to a chordal graph $G$. Specifically, we are interested in knowing whether the regularity and the projective dimension of such ideals can be expressed in terms of the combinatorial invariants of the associated hypergraphs, as mentioned above. The first main theorem along this direction is the following:

\begin{customthm}{\ref{reg_main}}
    Let $G$ be a chordal graph. Then for any $t\geq 2$,
    \[\reg(R/J_t(G))=(t-1)\nu_t(G),\]
    where $\nu_t(G)$ denotes the induced matching number of the hypergraph corresponding to $J_t(G)$.
\end{customthm}

\noindent As a corollary of the above theorem, we characterize when the ideal $J_t(G)$ corresponding to a chordal graph $G$ has a linear resolution as follows:

\begin{customcor}{\ref{cor-linear}}
        Let $G$ be a chordal graph and $t\ge 2$ be an integer. Then $J_{t}(G)$ has a linear resolution if and only if $G$ is $t$-gap-free (i.e., $\nu_{t}(G)=1$).
\end{customcor}

\noindent  
Next, we find the following combinatorial formula for the projective dimension of $R/J_t(G)$:

\begin{customthm}{\ref{pd_main}}
    Let $G$ be a chordal graph. Then for all $t\ge 2$, $\pd(R/J_t(G))=\b(J_t(G))$.
\end{customthm}
\noindent As an application of the above theorem, we combinatorially characterise when the ideal $J_t(G)$ corresponding to a chordal graph $G$ is Cohen-Macaulay. This ensures that the Cohen-Macaulay property of such ideals does not depend on the characteristic of the base field.

\begin{customcor}{\ref{cor-CM}}
        Let $G$ be a chordal graph and $t\ge 2$ be an integer. Then $J_{t}(G)$ is Cohen-Macaulay if and only if $J_{t}(G)$ is unmixed.
\end{customcor}

\noindent The above result partially generalize a famous theorem of Herzog-Hibi-Zheng \cite{HerzogHZ2006}, where they classified all Cohen-Macaulay chordal graphs. Furthermore, it follows from \Cref{reg_main} and \Cref{pd_main} that the regularity and projective dimension of the $t$-connected ideals of chordal graphs are independent of the characteristic of the base field (see \Cref{reg char} and \Cref{pd char}).

In the spirit of \Cref{reg_main} and \ref{pd_main}, one can try to obtain similar formulas for the regularity and projective dimension in the case of path ideals and clique ideals of graphs. We remark that an extensive amount of work is available in the literature on the $t$-path ideals of graphs (see \cite{2024pathideals, HangVu2024} and the references therein). Meanwhile, recently in \cite{2024pathideals}, we were able to show that the regularity and projective dimension formulas of edge ideals of chordal graphs in terms of the induced matching number and big height do not extend to $t$-path ideals for $t\geq 4$, even for the class of trees. Now, if one considers the $t$-clique ideal of a tree, then it is easy to see that the ideal is a zero ideal for $t\geq 3$. Regarding clique ideals of chordal graphs, we show that the above-mentioned formula of regularity cannot be extended to higher $t$ (\Cref{clique example}).

The paper is structured in the following way. In \Cref{prel}, we recall some standard notions and results of combinatorics and commutative algebra. In \Cref{sec reg}, we establish the regularity formula of $J_t(G)$ for chordal graphs and characterize when such ideals have linear resolutions. In \Cref{sec pd}, we derive the formula for the projective dimension of $J_t(G)$ in the case of chordal graphs and characterize when such ideals are Cohen-Macaulay. We make some concluding remarks in \Cref{sec remark}.

\section{Preliminaries}\label{prel}

In this section, we recall some preliminary notions from combinatorics and commutative algebra, which are used throughout the rest of the paper.

\subsection{Graph Theory and Combinatorics:}\label{subsection 2.1}

    A {\it graph} $G$ is a pair $(V (G), E(G))$, where $V(G)$ is called the \textit{vertex set} of $G$ and $E(G)$, a collection of subsets of $V(G)$ of size $2$, is known as the \textit{edge set} of $G$. We now recall some useful notation related to a graph $G$ that will be needed in the later sections.
    \begin{enumerate}
        \item If $x_1,\ldots,x_m\in V(G)$, then $G\setminus \{x_1,\ldots,x_m\}$ denotes the graph with the vertex set $V(G)\setminus \{x_1,\ldots,x_m\}$ and the edge set $\{\{u,v\}\in E(G)\mid x_i\notin \{u,v\}\text{ for each }i\in[m]\}$. If $m=1$, then $G\setminus\{x_1\}$ is simply denoted by $G\setminus x_1$.

        \item For $C\subseteq V(G)$, the set of {\it neighbors} of $C$, denoted by $N_G(C)$, is the set 
        $$\{w\in V(G)\setminus C \mid \{w,x\}\in E(G)\text{ for some }x\in C\}.$$
        The set of {\it closed neighbors} of $C$ is the set $N_G(C)\cup C$ and is denoted by $N_G[C]$. If $C=\{a\}$ for some $a\in V(G)$, then we simply denote the sets $N_G(\{a\})$ and $N_G[\{a\}]$ as $N_G(a)$ and $N_G[a]$, respectively.

        \item Let $W\subseteq V(G)$. Then the {\it induced subgraph of $G$ on $W$}, denoted by $G[W]$, is the graph on the vertex set $W$ with the edge
        set $\{e \in E(G) ~|~ e \subseteq W\}$. Note that $G[W]=G\setminus (V(G)\setminus W)$.

    \item Let $t\ge 2$ be an integer, and 
    $$\quad\quad \U=\{C_1,\ldots ,C_r : |C_i|=t, G[C_i] \text{ is connected, } C_i\cap C_j=\emptyset \text{ for all } 1\leq i<j\leq r \}.$$ 
    We say that $\U$ is a \textit{$t$-induced matching} of G if $E(G[\cup_{i=1}^rC_i])=\cup_{i=1}^rE(G[C_i])$. The \textit{$t$-induced matching number} of $G$, denoted by $\nu_t(G)$, is given by $$\nu_t(G)=\max\{|\U|:\U \text{ is a }t\text{-induced matching of }G\}.$$ 
    If $t=2$, then $\nu_2(G)$ is the usual {\it induced matching number} of $G$ and is also denoted in the literature by $\nu(G)$. We say $G$ is \textit{$t$-gap-free} whenever $\nu_{t}(G)=1$.   
    \end{enumerate}

\newpage
    
\noindent   
\textbf{Various classes of simple graphs:} 
\begin{enumerate}
    \item A {\it path} graph $P_n$ of length $n-1$ is a graph with the vertex set $\{x_1,\ldots,x_{n}\}$, and the edge set $\{\{x_i,x_{i+1}\}\mid 1\le i\le n-1\}$. A {\it cycle} $C_n$ of length $n$ is a graph with the vertex set $\{x_1,\ldots,x_n\}$, and the edge set $\{\{x_1,x_n\},\{x_i,x_{i+1}\}\mid 1\le i\le n-1\}$.

    \item For a positive integer $m$, a \textit{complete graph} $K_{m}$ is a graph on $m$ vertices such that there is an edge between any two distinct vertices.

    \item A graph $G$ is called \textit{chordal} if $G$ does not contain any induced cycle of length more than three. If $G$ is a chordal graph, then $G$ contains at least one vertex $x$ such that $N_G(x)$ is a complete graph (see \cite{Dirac}). Such a vertex is called a {\it simplicial vertex} of $G$. Note that any induced subgraph of a chordal graph is again a chordal graph.
    \end{enumerate}

\subsection{Stanley-Reisner ideal of $\mathrm{Ind}_t(G)$:} Let $R$ denote the polynomial ring $\mathbb K[x_1,\ldots,x_n]$, where $\mathbb K$ is a field. For $F\subseteq \{x_1,\ldots,x_n\}$, let $\x_{F}$ denote the monomial $\prod_{x_i\in F}x_i$. Recall that a {\it simplicial complex} on the vertex set $V(\Delta)=\{x_1,\ldots,x_n\}$ is a collection of subsets of $\Delta$ satisfying the following two properties: (i) each $x_i\in\Delta$ for $i\in[n]$; (ii) if $A\in\Delta$ and $B\subseteq A$, then $B\in\Delta$. Now, given such a simplicial complex $\Delta$, the Stanley-Reisner ideal $I_{\Delta}$ of $\Delta$ is the monomial ideal $\l \x_F\mid F\notin\Delta\r$. Now take $G$ to be a graph on the vertex set $V(G)=\{x_1,\ldots,x_n\}$. Then the Stanley-Reisner ideal of $\mathrm{Ind}_t(G)$, denoted by $J_t(G)$, is the ideal 
\[
J_t(G)=\left\l \x_C\mid C\subseteq V(G), |C|=t,\text{ and }G[C]\text{ is connected }\right\r
\]
in the polynomial ring $R$. The ideal $J_t(G)$ is referred to as the {\it $t$-connected ideal} of $G$ in \cite{AJM2024}.

\subsection{$J_t(G)$ as edge ideal of a hypergraph:} 
A {\it hypergraph} $\H$ is a pair $(V(\H),E(\H))$, where $E(\H)\subseteq 2^{V(\H)}$, and for any two $\E_1,\E_2\in E(\H)$, $\E_1\not\subset \E_2$. The sets $V(\H)$ and $E(\H)$ are called the \textit{vertex set} and \textit{edge set} of $\H$, respectively. For a fixed positive integer $m$, if $|\E|=m$ for each $\E\in E(\H)$, then we say that $\H$ is an {\it $m$-uniform hypergraph}. Note that if $\H$ is a $2$-uniform hypergraph, then $\H$ is just a graph. As in the case of graphs, if $A\subseteq V(\H)$, then $\H\setminus A$ denotes the hypergraph with the vertex set $V(\H)\setminus A$, and the edge set $\{\E\in E(\H)\mid \E\cap A=\emptyset\}$. Similarly, for any $A\subseteq V(\H)$, the hypergraph $\H\setminus (V(\H)\setminus A)$ is called the \textit{induced subhypergraph} of $\H$ on the vertex set $A$. For $x\in V(\H)$, we simply write $\H\setminus x$ to denote the hypergraph $\H\setminus\{x\}$. A subset $U\subseteq V(\H)$ is called a \textit{vertex cover} of $\H$ if for any edge $\E\in E(\H)$ one has $\E\cap U\neq \emptyset$. A \textit{minimal vertex cover} of $\H$ is a vertex cover that is minimal with respect to inclusion. 

Let $\H$ be a hypergraph on the vertex set $\{x_1,\ldots,x_n\}$ and let $R=\K[x_1,\ldots,x_n]$. Corresponding to each $\E\in E(\H)$, one can assign the monomial $\x_{\E}=\prod_{x_j\in \E}x_j$ in $R$. Then the ideal $\l \x_{\E}\mid \E\in E(\H) \r$ is called the {\it edge ideal} of $\H$, and is denoted by $I(\H)$. Let $I\subseteq R$ be a square-free monomial ideal with the unique minimal monomial generating set $\G(I)$. Then $I$ can be viewed as an edge ideal of a hypergraph $\H_{I}$, where $V(\H_{I})=\{x_1,\ldots,x_n\}$ and $E(\H_{I})=\{\{x_{i_1},\ldots,x_{i_r}\}\mid x_{i_1}\cdots x_{i_r}\in \G(I)\}$. In other words, we have $I=I(\H_{I})$. It is well-known in the literature that the minimal prime ideals of $I$ (equivalently, the associated primes of $I$ since $I$ is a radical ideal) are exactly the ideals generated by the minimal vertex covers of $\H_{I}$. Consequently, the \textit{height} of $I$ (resp., the \textit{big height} of $I$), denoted by $\h(I)$ (resp., $\b(I)$), is the minimum (resp., maximum) cardinality of a minimal vertex cover of~ $\H_{I}$.\par 

Let $G$ be a graph on the vertex set $\{x_1,\ldots,x_n\}$. Consider the ideal $J_{t}(G)$ in the polynomial ring $R$. Since $J_{t}(G)$ is a square-free monomial ideal, from the previous discussion, we can associate a hypergraph, say $\H(G,t)$, on the vertex set $\{x_1,\ldots,x_n\}$ such that $J_t(G)=I(\H(G,t))$. More precisely,
\begin{enumerate}
    \item[$\bullet$] $V(\H(G,t))=V(G)$,
    \item[$\bullet$] $E(\H(G,t))=\{\{x_{i_1},\ldots,x_{i_t}\}\subseteq V(G)\mid G[\{x_{i_1},\ldots,x_{i_t}\}] \text{ is connected }\}$.
\end{enumerate}

\subsection{Some algebraic invariants:}
Let $I$ be a graded ideal in $R=\K[x_1,\ldots, x_n]$, where $\mathbb K$ is a field. Then, a graded minimal free resolution of $R/I$ is an exact sequence
\[
\mathcal F_{\cdot}: \,\, 0\rightarrow F_r\xrightarrow{\partial_{r}}\cdots\xrightarrow{\partial_{2}} F_1\xrightarrow{\partial_1} F_0\xrightarrow{\partial_0} R/I\rightarrow 0, 
\]
where $F_0=R$, $F_i=\oplus_{j\in\mathbb N}R(-j)^{\beta_{i,j}(R/I)}$ for each $i\ge 1$, $\partial_0$ is the natural quotient map, and $R(-j)$ is the polynomial ring $R$ with its grading twisted by $j$. The numbers $\beta_{i,j}(R/I)$ are uniquely determined, and are called the $i^{th}$ $\mathbb N$-graded {\it Betti numbers} of $R/I$ in degree $j$. The {\it Castelnuovo-Mumford regularity} (or simply called the {\it regularity}) of $R/I$, denoted by $\reg(R/I)$, is the number $\max\{j-i\mid \beta_{i,j}(R/I)\neq 0\}$. The invariant $\max\{i\mid \beta_{i,j}(R/I)\neq 0\}$ is called the {\it projective dimension} of $R/I$, and is denoted by $\pd(R/I)$. Let $I$ be a graded ideal generated in a single degree $r$. Then, we say that $I$ has an {\it $r$-linear resolution} (or simply, a {\it linear resolution}) if $\reg(R/I)=r-1$.
\par 

The following are some well-known results regarding regularity and projective dimension, which we are going to use in the subsequent sections.

\begin{lemma}\label{reg and pd sum}\cite[cf. Lemma 2.5]{HaTrungTrung}
        Let $I_1\subseteq R_1=\mathbb K[x_1,\ldots,x_n]$ and $I_2\subseteq R_2=\mathbb K[y_1,\ldots, y_m]$ be two graded ideals. Consider the ideal $I=I_1R+I_2R\subseteq R=\mathbb K[x_1,\ldots,x_n,y_1,\ldots,y_m]$. Then 
        \begin{enumerate}
            \item[(i)] $\reg(R/I)=\reg(R_1/I_1)+\reg(R_2/I_2)$,
            \item[(ii)] $\pd(R/I)=\pd(R_1/I_1)+\pd(R_2/I_2).$
        \end{enumerate}
    \end{lemma}

\begin{lemma}\cite[Lemma 2.10, Lemma 5.1]{DHS}\label{regularity and pd lemma}
    Let $I\subseteq R$ be a square-free monomial ideal and let $x_i$ be a variable appearing in some generator of $I$. Then
    \begin{enumerate}
        \item[(i)] $   \reg(R/I)\le\max\{\reg(R/(I:x_i))+1,\reg(R/\langle I,x_i\rangle)\}$. Moreover,\\ $\reg(R/I)\in\{\reg(R/(I:x_i))+1,\reg(R/\langle I,x_i\rangle)\}$. 
        \item[(ii)] $\pd(R/I)\le\max\{\pd(R/(I:x_i)),\pd(R/\langle I,x_i\rangle)\}.$
    \end{enumerate}
    \end{lemma}

    \begin{lemma}\label{regularity and pd lemma intersection}\textup{(cf. \cite[Chapter 18]{IPBook})}
    Let $J$ and $K$ be two graded ideals of $R$. Then 
    \begin{enumerate}
        \item[(i)] $\reg(R/(J+K))\le \max\{\reg(R/J),\reg(R/K),\reg(R/(J\cap K))-1\},$
        \item[(ii)] $  \pd(R/(J+K))\le \max\{\pd(R/J),\pd(R/K),\pd(R/(J\cap K))+1\}.$
    \end{enumerate}
\end{lemma}

\subsection{Bounds on regularity and projective dimension:}\label{induced_matching}
Let $\H$ be a hypergraph. A \textit{matching} in $\H$ is a collection of pairwise disjoint edges of $\H$. More precisely, a subset $\mathcal D\subseteq E(\H)$ is called a matching of $\H$ if for any two distinct edges $\E_1,\E_2\in \mathcal D$, one has $\E_1\cap \E_2=\emptyset$. An \textit{induced matching} is a matching $\mathcal D$ in $\H$ such that the edge set of the induced subhypergraph of $\H$ on the vertices of $\mathcal D$ is precisely the set $\mathcal D$. The following lower bound on the regularity in terms of the induced matching is well-known.

\begin{lemma}\cite[Corollary 3.9]{MoreyVillarreal}\label{reg induced matching}
    Let $\H$ be a hypergraph, and $\mathcal D$ an induced matching of $\H$. Then \[\reg(R/I(\H))\geq \sum_{\E\in \mathcal D}(|\E|-1).\]
\end{lemma}

\noindent

Let us define $\nu(\H)=\max\{|\mathcal D|: \mathcal D \text{ is an induced matching of }\H\}$, and call this the \textit{induced matching number} of the hypergraph $\H$. Then for a simple graph $G$, $\nu(G)$ gives a crude lower bound of $\reg(R/I(G))$. For our purpose, given a simple graph $G$, we call an induced matching of the hypergraph $\H(G,t)$ a \textit{$t$-induced matching of $G$}. Observe that $\{\E_1,\ldots,\E_r\}\subseteq E(\H)$ is an induced matching of $\H(G,t)$ if and only if $|\E_i|=t$, $G[\E_i]$ is connected for each $i\in[r]$, $\E_i\cap \E_j=\emptyset$ for all $1\le i<j\le r$, and $E(G[\cup_{i=1}^r\E_i])=\cup_{i=1}^rE(G[\E_i])$. Thus, we have $\nu(\H(G,t))=\nu_t(G)$, where $\nu_t(G)$ is defined as in \Cref{subsection 2.1}. Consequently, in our case, we have the following lower bound for the regularity of $J_t(G)$.
\begin{lemma}\label{lower bound}
    Let $G$ be a finite simple graph. Then $\reg(R/J_t(G))\ge (t-1)\nu_t(G)$.
\end{lemma}

Using \cite[Proposition 1.2.13]{BrunsHerzog} and the Auslander-Buchsbaum formula \cite[Theorem 1.3.3]{BrunsHerzog}, one can get an analogous bound for the projective dimension of $R/J_t(G)$ in terms of $\b(J_t(G))$ as follows:

 \begin{lemma}\label{pd and bight}
    Let $G$ be a finite simple graph. Then $\pd(R/J_t(G))\ge \b(J_t(G))$.
\end{lemma}

\noindent \textbf{Note:} Let $s$ be the maximum cardinality of the set of vertices in a connected component of a graph $G$. Then $J_t(G)=\l 0\r$ for all $t>s$. Thus, it is enough to focus on non-zero $J_t(G)$, and sometimes we will assume this without mentioning it explicitly. 

\section{Castelnuovo-Mumford Regularity and Linearity}\label{sec reg}

In this section, we compute the regularity of $J_t(G)$ for a chordal graph $G$ in terms of the $t$-induced matching number of $G$. As a consequence, we characterize when such an ideal has a linear resolution. Let us start with the following easy observation.

\begin{lemma}\label{colon comma exchange}
    Let $I\subseteq R$ be a monomial ideal, and let $x_1,\ldots,x_r\in R$ be some indeterminates. Then \[ (I:x_r)+\l x_1,\ldots,x_{r-1} \r=(( I+\l x_1,\ldots,x_{r-1} \r):x_r).\]
\end{lemma}

\begin{lemma}\label{simplicial connected}
    Let $G$ be a connected graph and $x\in V(G)$ be a simplicial vertex. If $A\subseteq V(G)$ such that $x\in A$, $|A|=t\geq 2$, and $G[A]$ is connected, then $G[A\setminus \{x\}]$ is also connected.
\end{lemma}    
    \begin{proof}
        If possible, let us assume that $G[A\setminus \{x\}]$ is a disconnected graph. Let $H_1$ and $H_2$ be any two connected components of $G[A\setminus \{x\}]$. Since $G[A]$ is connected, there are $y_1\in V(H_1)$ and $y_2\in V(H_2)$ such that $\{x,y_1\},\{x,y_2\}\in E(G[A])$. Then $\{y_1,y_2\}\in E(G[A])$, as $x$ is a simplicial vertex. This is a contradiction to the fact that $H_1$ and $H_2$ are connected components in $G[A\setminus \{x\}]$. 
    \end{proof}

\begin{lemma}\label{induced match lemma}
    Let $G$ be a graph, and $C\subseteq V(G)$ such that $G[C]$ is connected, $|C|=t-1$, and $w\in N_G(C)$. Then \[\nu_t(G\setminus (N_G[C]\cup N_G[w]))\leq \nu_t(G)-1.\]
    \end{lemma}
    \begin{proof}
        Let $\U$ be a $t$-induced matching of $G\setminus (N_G[C]\cup N_G[w])$ such that 
        $$|\U|=\nu_t(G\setminus (N_G[C]\cup N_G[w])).$$ 
        Since $w\in N_G(C)$ and $G[C]$ is connected, the induced subgraph $G[C\cup\{w\}]$ is connected. Moreover, $G[C\cup\{w\}]$  is a connected component of the the induced subgraph of $G$ on the vertex set $V(G\setminus (N_G[C]\cup N_G[w]))\cup C\cup \{w\}$. Therefore, $\U\cup \{C\cup\{w\}\}$ is a $t$-induced matching of $G$, and hence, the inequality follows. 
    \end{proof}

\noindent \textbf{Notations:} We adopt the following notations for the rest of the paper.

\begin{enumerate}
    \item[(i)] For a hypergraph $\H$ and some given subsets $C_{1},\ldots,C_r$ of the vertex set $V(\H)$, we define a new hypergraph, denoted by $\H\setminus \{C_1,\ldots,C_r\}$, by deleting some particular edges as follows:
    \begin{align*}
        V(\H\setminus \{C_1,\ldots,C_r\})&=V(\H);\\
        E(\H\setminus \{C_1,\ldots,C_r\})&= E(\H)\setminus\{\E\in E(\H)\mid C_i\subseteq \E\text{ for some }1\le i\le r\}.
    \end{align*}
    
\item[(ii)] Let $x$ be a vertex of the graph $G$, and $t\ge 2$ be an integer. Define the set 
$$\A_x:=\{C\subseteq V(G): |C|=t-1,x\in C, G[C] \text{ is connected}\}.$$ Without loss of generality, let $\A_x=\{C_1,\ldots ,C_k\}$. Let us define $\mathcal B_{C_1}:=N_G(C_1)$. Then for each $2\leq i\leq k$, we define 
    \[
    \B_{C_i}:=\{w\in N_G(C_i) \mid C_i\cup \{w\}\neq C_j\cup \{w'\} \text{ for any } w'\in N_G(C_j), \text{ where }1\leq j\leq i-1\}.
    \]
    By construction, $\B_{C_1}\neq\emptyset$ when $J_t(G)\neq \l 0\r$.
\end{enumerate}

The next lemma plays a crucial role in establishing the main results of this article. \medskip

\begin{lemma}\label{main lemma}
    Let $G$ be a chordal graph and $x\in V(G)$ be a simplicial vertex. Let us consider the set $\A_x=\{C_1,\ldots ,C_k\}$. For $1\leq i\leq k$, define 
    \begin{align*}
        J_i&:=\x_{C_i}\langle w\mid w\in \B_{C_i}\rangle,\\
        K_i&:= I(\H(G,t)\setminus \{C_1,\ldots,C_i\}).
    \end{align*}
    If $\B_{C_i}\neq \emptyset$, then 
    \begin{enumerate}
        \item  $J_i+K_i=I(\H(G,t)\setminus \{C_1,\ldots,C_{i-1}\})$,
        \item $J_i\cap K_i=\x_{C_i}L_i$, where $L_i=\left\l \frac{\lcm(m,m')}{\x_{C_i}}\mid m\in J_i\text{ and }m'\in K_i\right\r$. Furthermore, for any $w\in \B_{C_i}$, we have $(L_i:w)= M_i+N_i+Q_i$, where 
        \begin{align*}
            M_i& =\langle v\mid v\in N_G(C_i)\setminus \{w\}\rangle, \\
            N_i&=\langle v\mid v\in N_G(w)\setminus N_G[C_i]\rangle, \\
            Q_i&=\langle \x_{C}\mid C\subseteq V(G), |C|=t, G[C] \text{ is connected }, C\cap (N_G[C_i]\cup N_G[w])=\emptyset \rangle.
        \end{align*}
    \end{enumerate}

    \begin{proof}
        (1) Follows immediately from the construction of the ideals $J_i$ and $K_i$.

        \noindent
        (2) Clearly, we have $J_i\cap K_i=\x_{C_i}L_i$. Now, fix any $1\leq i\leq k$ and any $w\in \B_{C_i}$. We first show that $M_i+N_i+Q_i\subseteq (L_i:w)$. Let $v\in N_G(C_i)\setminus \{w\}$. Then we can write $wv\x_{C_i}=\L(w\x_{C_i}, wv\x_{C_i\setminus \{x\}})$. Since $w,v\in N_G(C_i)$, we see that $G[C_i\cup \{w,v\}]$ is a connected graph. Hence by \Cref{simplicial connected}, $G[(C_i\setminus \{x\})\cup \{w,v\}]$ is also connected. Also, observe that $C_j\nsubseteq (C_i\setminus \{x\})\cup \{w,v\}$ for any $1\leq j\leq i$, as $x\in C_j$ for all such $j$. Hence we get $wv\x_{C_i\setminus \{x\}}\in K_i$. Moreover, from our choice of $w\in \B_{C_i}$, it is easy to see that $w\x_{C_i}\in J_i$. Thus, we have $wv\x_{C_i}\in J_i\cap K_i$, and therefore, $v\in (L_i:w)$. Now, let $v\in N_G(w)\setminus N_G[C_i]$. Then again, we can write $wv\x_{C_i}=\L(w\x_{C_i}, wv\x_{C_i\setminus \{x\}})$. By similar arguments as above, we obtain $w\x_{C_i}\in J_i$ and $wv\x_{C_i\setminus \{x\}}\in K_i$, and thus, $wv\x_{C_i}\in J_i\cap K_i$, which in turn, gives $v\in (L_i:w)$. Finally, let $\x_C\in Q_i$, where $C\subseteq V(G)$ is such that $|C|=t$, $G[C]$ is connected, and $C\cap (N_G[C_i]\cup N_G[w])=\emptyset$. Then we write \[w\x_{C_i}\x_C=\L(w\x_{C_i}, \x_C),\]
        where $w\x_{C_i}\in J_i$ and $\x_C\in K_i$. Thus, $w\x_{C_i}\x_C\in J_i\cap K_i$ and hence $\x_C\in (L_i:w)$.

        We now proceed to show that $(L_i:w)\subseteq M_i+N_i+Q_i$. Let $A\subseteq V(G)$ be such that $\x_A\in \G(K_i)$. We consider the following two cases:
        
        \noindent
        \textbf{Case-I.} Let $A\cap N_G[C_i]=\emptyset$. Then $w\notin A$ and hence $\L(w\x_{C_i},\x_A)=w\x_{C_i}\x_A$. Now if $A\cap N_G[w]\neq \emptyset$, then $A\cap N_G[w]\subseteq N_G(w)\setminus N_G[C_i]$, since $w\notin A$. Also if $A\cap N_G[w]= \emptyset$, then clearly $A\cap (N_G[C_i]\cup N_G[w])=\emptyset$. Thus, in any case, we have 
        \[(\L(w\x_{C_i},\x_A):w\x_{C_i})=\x_A\in N_i+Q_i.\]

        \noindent
        \textbf{Case-II.} Let $A\cap N_G[C_i]\neq \emptyset$. Note that, since $A\nsubseteq C_i$, and $G[C_i\cup A]$ is connected, there exists some $a\in A\setminus C_i$ and $b\in C_i$ such that $\{a,b\}\in E(G)$. Thus $A\cap N_G(C_i)\neq \emptyset$.
        
        \noindent
        First, consider the case when $A\cap N_G(C_i)=\{w\}$. Recall that $|C_i|=t-1$, and $|A|=t$. Since $C_i\nsubseteq A$, there exists some $v\in A\setminus C_i$ such that $v\neq w$ and $v\notin N_G[C_i]$. Since $G[A]$ is connected, there exists a shortest path $v=y_0,y_1,\ldots , y_r=w$ such that $y_i\in A$ for all $0\leq i\leq r$. Note that $y_{r-1}\in N_G(w)$, and $v\notin N_G[C_i]$. So if $y_{r-1}\in C_i$, then there exists $1\leq i\leq r-2$ such that $y_i\in A\cap N_G(C_i)$, which is a contradiction to the fact that $A\cap N_G(C_i)=\{w\}$. Thus $y_{r-1}\notin N_G[C_i]$ and hence $y_{r-1}\in N_G(w)\setminus N_G[C_i]$. Therefore, in this case $(\L(w\x_{C_i},\x_A):w\x_{C_i})\in N_i$. 

        \noindent
        Finally, let us assume that there exists $w'\in A\cap N_G(C_i)$, where $w'\neq w$. Then, we have $(\L(w\x_{C_i},\x_A):w\x_{C_i})\in M_i$. Therefore, $(L_i:w)=M_i+N_i+Q_i$.
    \end{proof}
\end{lemma}

In the following example, we illustrate some of the notations used in \Cref{main lemma} with the aid of \Cref{chordal G}.
\begin{example}
    \begin{figure}[h!]
        \centering
         \begin{tikzpicture}
    [scale=1]
\draw [fill] (0,0) circle [radius=0.08];
\draw [fill] (0,2) circle [radius=0.08];
\draw [fill] (2,0) circle [radius=0.08];
\draw [fill] (2,2) circle [radius=0.08];
\draw [fill] (3,1) circle [radius=0.08];
\draw [fill] (4,1) circle [radius=0.08];
\draw [fill] (5,2) circle [radius=0.08];
\draw [fill] (5,0) circle [radius=0.08];
\draw [fill] (6,1) circle [radius=0.08];
\draw [fill] (7,1) circle [radius=0.08];
\draw [fill] (8,1) circle [radius=0.08];
\draw [fill] (9,2) circle [radius=0.08];
\draw [fill] (9,0) circle [radius=0.08];
\draw [fill] (10,1) circle [radius=0.08];
\node at (0,-0.5) {$x_1$};
\node at (0,2.5) {$x_2$};
\node at (2,2.5) {$x_4$};
\node at (2,-0.5) {$x_3$};
\node at (2.5,1) {$x_5$};
\node at (4,1.5) {$x_6$};
\node at (5,2.5) {$x_7$};
\node at (5,-0.5) {$x_8$};
\node at (6,1.5) {$x_9$};
\node at (7,1.5) {$x_{10}$};
\node at (8,1.5) {$x_{11}$};
\node at (9,2.5) {$x_{12}$};
\node at (9,-0.5) {$x_{13}$};
\node at (10,1.5) {$x_{14}$};
\draw (0,0)--(0,2)--(2,2)--(2,0)--(0,0)--(2,2);
\draw (0,2)--(2,0)--(3,1)--(2,2);
\draw (3,1)--(4,1)--(5,2)--(5,0)--(4,1);
\draw (5,2)--(6,1)--(5,0);
\draw (6,1)--(7,1)--(8,1)--(9,2)--(9,0)--(10,1)--(8,1)--(9,0);
\draw (9,2)--(10,1);
\draw (2,2)--(4,1)--(2,0);
\end{tikzpicture}
        \caption{A chordal graph $G$.}
        \label{chordal G}
    \end{figure}
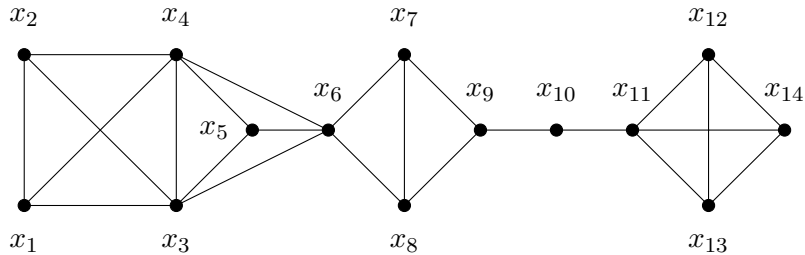
{\rm Let $G$ be the chordal graph as in \Cref{chordal G} and $x_5$ a simplicial vertex in $G$. Consider the ideal $I(\H(G,4))$ inside the polynomial ring $R=\K[x_1,\ldots,x_{14}]$. In this case, observe that $\mathcal{A}_{x_5}=\{C_1,C_2,\ldots,C_9\}$, where 
\begin{align*}C_1=\{x_3,x_4,x_5\},C_2=\{x_3,x_5,x_6\},C_3=\{x_4,x_5,x_6\},&C_4=\{x_2,x_4,x_5\},
    C_5=\{x_1,x_4,x_5\},\\
    C_6=\{x_2,x_3,x_5\}, C_7=\{x_1,x_3,x_5\}, & C_8=\{x_5,x_6,x_7\}, C_9=\{x_5,x_6,x_8\}.
\end{align*}
 It is easy to see that $\B_{C_1}=\{x_1,x_2,x_6\}$. Thus, $J_1=\l x_1x_3x_4x_5,x_2x_3x_4x_5,x_3x_4x_5x_6\r$ and the monomial generators of $K_1$ consist of all monomial generators of $I(\H(G,4))$ excluding the generators of $J_1$. Note that, $x_4\notin\B_{C_2}$ since $\{x_4\}\cup C_2=\{x_6\}\cup C_1$. Consequently, we have $\B_{C_2}=\{x_1,x_2,x_7,x_8\}$. Thus $J_{2}=\l x_1x_3x_5x_6,x_2x_3x_5x_6,x_3x_5x_6x_7,x_3x_5x_6x_8\r$ and the generators of $K_2$ consist of all monomial generators of $I(\H(G,4))$ excluding the monomial generators of $J_1$ and $J_2$. Similarly, one can determine the sets $\B_{C_i}$ and the ideals $J_{C_i}$, for $3\le i\le 9$.}
\end{example}

We are now ready to prove the main theorem of this section.

\begin{theorem}\label{reg_main}
    Let $G$ be a chordal graph. Then for any $t\geq 2$,
    \[\reg(R/J_t(G))=(t-1)\nu_t(G).\]
    \end{theorem}
    \begin{proof}
        In view of \Cref{reg induced matching}, it is enough to prove that $\reg(R/J_t(G))\leq (t-1)\nu_t(G)$. We prove this using induction on $|V(G)|$. First and foremost, if $|V(G)|\leq t$, then one can see that either $J_t(G)=\langle \prod_{x\in V(G)}x\rangle$ or $J_t(G)=\l 0\r$. In either case, it is easy to see that  $\reg(R/J_t(G))=(t-1)\nu_t(G)$. Therefore, we may assume that $|V(G)|\geq t+1$. Also, we can assume that $\nu_t(G)\geq 1$. Moreover, we will write $J_t(G)=I(\H(G,t))$, where $\H(G,t)$ is the hypergraph corresponding to the ideal $J_t(G)$. Now, let $x\in V(G)$ be a simplicial vertex of $G$, and $\A_x=\{C_1,\ldots , C_k\}$. Following \Cref{main lemma}, whenever $\B_{C_i}\neq \emptyset$ for some $i\in\{1,\ldots,k\}$, we denote  
        \begin{align*}
            J_i&=\x_{C_i}\langle w\mid w\in \B_{C_i}\rangle,\\
            K_i&=I(\H(G,t)\setminus \{C_1,\ldots,C_i\}),\\
            J_i&\cap K_i=\x_{C_i}L_i.
        \end{align*}

        \noindent \textbf{Claim:} $\reg(R/L_i)\leq (t-1)\nu_t(G)-(t-2)$ for each $1\leq i\leq k$.

        \noindent
        \textit{Proof of the claim.} Let $\B_{C_i}=\{w_1,\ldots , w_s\}$. Then $L_i+\langle w_1,\ldots , w_s\rangle= \langle w_1,\ldots, w_s\rangle$ by \Cref{main lemma}. Thus, 
        \[\reg(R/(L_i+\langle w_1,\ldots , w_s\rangle))=0\leq (t-1)\nu_t(G)-(t-2).\]
        Now by \Cref{main lemma}, 
        \begin{align*}
            ((L_i+\langle w_1,\ldots , w_{s-1}\rangle):w_s)&=(L_i:w_s)+\langle w_1,\ldots , w_{s-1}\rangle\\
            &= \langle v\mid v\in N_G(C_i)\setminus \{w_s\}\rangle+\langle v\mid v\in N_G(w_s)\setminus N_G[C_i]\rangle \\
            &\hspace{2em}+\langle \x_{C}\mid C\subseteq V(G), |C|=t, G[C] \text{ is connected },\\ &\hspace{2cm} C\cap (N_G[C_i]\cup N_G[w_s])=\emptyset \rangle\\
            &=\langle v\mid v\in N_G(C_i)\setminus \{w_s\}\rangle+\langle v\mid v\in N_G(w_s)\setminus N_G[C_i]\rangle \\
            &\hspace{2em}+ J_t(G\setminus (N_G[C_i]\cup N_G[w_s])).
        \end{align*}
        Then 
        \begin{align*}
            \reg(R/((L_i+\langle w_1,\ldots , w_{s-1}\rangle):w_s))&= \reg(R/J_t(G\setminus (N_G[C_i]\cup N_G[w_s])))\\
            &\leq (t-1)\nu_t(G\setminus (N_G[C_i]\cup N_G[w_s]))\\
            &\leq (t-1)(\nu_t(G)-1),
        \end{align*}
        where the first inequality is by the induction hypothesis and the second inequality follows from \Cref{induced match lemma}. Hence, by \Cref{regularity and pd lemma}, we have $\reg(R/(L_i+\langle w_1,\ldots , w_{s-1}\rangle)\leq (t-1)\nu_t(G)-(t-2).$ Now for each $2\leq j\leq s-1$, similarly using \Cref{main lemma}, we have
        \begin{align*}
            ((L_i+\langle w_1,\ldots , w_{s-j}\rangle):w_{s-j+1})&=(L_i:w_{s-j+1})+\langle w_1,\ldots , w_{s-j}\rangle\\
            &= \langle v\mid v\in N_G(C_i)\setminus \{w_{s-j+1}\}\rangle+\langle v\mid v\in N_G(w_{s-j+1})\setminus N_G[C_i]\rangle \\
            &\hspace{2em}+\langle \x_{C}\mid C\subseteq V(G), |C|=t, G[C] \text{ is connected },\\ &\hspace{2cm} C\cap (N_G[C_i]\cup N_G[w_{s-j+1}])=\emptyset \rangle\\
            &=\langle v\mid v\in N_G(C_i)\setminus \{w_{s-j+1}\}\rangle+\langle v\mid v\in N_G(w_{s-j+1})\setminus N_G[C_i]\rangle \\
            &\hspace{2em}+ J_t(G\setminus (N_G[C_i]\cup N_G[w_{s-j+1}])).
        \end{align*}
        Thus for each $2\leq j\leq s-1$, we get
        \begin{align*}
            \reg(R/((L_i+\langle w_1,\ldots , w_{s-j}\rangle):w_{s-j+1}))&= \reg(R/J_t(G\setminus (N_G[C_i]\cup N_G[w_{s-j+1}])))\\
            &\leq (t-1)\nu_t(G\setminus (N_G[C_i]\cup N_G[w_{s-j+1}]))\\
            &\leq (t-1)(\nu_t(G)-1),
        \end{align*}
        Therefore, repeatedly applying \Cref{regularity and pd lemma}, we obtain 
        $$\reg(R/L_i)\leq (t-1)(\nu_t(G)-1)-1=(t-1)\nu_t(G)-(t-2).$$
        This completes the proof of the claim.\par 
        
        Now, consider the ideal $K_k=I(\H(G,t)\setminus \{C_1,\ldots,C_k\})=I(\H(G\setminus x,t))=J_t(G\setminus x)$.  By the induction hypothesis, $\reg(R/K_k)\leq (t-1)\nu_t(G\setminus x)\leq (t-1)\nu_t(G)$. Also, observe that $\reg(R/J_k)=t-1\leq (t-1)\nu_t(G)$. Moreover, we have 
        $$\reg(R/(J_k\cap K_k))=\reg(R/\x_{C_i}L_k)\leq (t-1)+(t-1)\nu_t(G)-t-2)=(t-1)\nu_t(G)+1.$$ 
        Hence, by \Cref{regularity and pd lemma intersection}, $\reg(R/(J_k+K_k))\leq (t-1)\nu_t(G)$. Note that the ideal $J_k+K_k$ is nothing but $I(\H\setminus \{C_1,\ldots,C_{k-1}\})=K_{k-1}$. We now write $I(\H\setminus \{C_1,\ldots,C_{k-2}\})=J_{k-1}+K_{k-1}$ and continue the above process. Note that if for some $1\leq i\leq k$, $\B_{C_i}=\emptyset$, then $K_i=K_{i-1}$. Hence, after a finite number of steps, we obtain 
        \[\reg(R/J_t(G))=\reg(R/I(\H(G,t)))=\reg(R/(J_1+K_1)\leq (t-1)\nu_t(G),\]
        and this completes the proof.
    \end{proof}

    \begin{remark}\label{reg char}
        Since the regularity of $t$-connected ideals of chordal graphs can be expressed in terms of a combinatorial invariant of the graph, it follows that in this case the regularity is independent of the characteristic of the base field.
    \end{remark}

As an application of \Cref{reg_main}, we get a complete classification of chordal graphs $G$ such that $J_t(G)$ has a linear resolution as follows.

\begin{corollary}\label{cor-linear}
        Let $G$ be a chordal graph and $t\ge 2$ be an integer. Then $J_{t}(G)$ has a linear resolution if and only if $G$ is $t$-gap-free (i.e., $\nu_{t}(G)=1$).
\end{corollary}

\begin{remark}\normalfont
    The result stated in the above corollary was proved earlier in \cite{AJM2024}. In fact, more generally, it was shown in \cite[Theorem 5.1]{AJM2024} that a chordal graph $G$ is $t$-gap-free if and only if $J_t(G)$ has linear quotients.    
\end{remark}


\begin{example}{\rm
    Let us consider the graph $G$ as shown in \Cref{chordal G}. Then one can deduce that 
    \begin{align*}
        \nu_{t}(G)=\begin{cases}
            4 \quad \text{ for } t=2,\\
            3 \quad \text{ for } t=3,\\
            2 \quad \text{ for } t=4,5,6,\\
            1 \quad \text{ for } t=7,\ldots,14,\\
            0 \quad\text{ for } t>14.
        \end{cases}
    \end{align*}
    Therefore, using \Cref{reg_main}, we can derive $\reg(R/J_t(G))$ for all $t\geq 2$. Note that $J_t(G)=\l 0\r$ for all $t>14$. If $J_t(G)\neq \l 0\r$, then due to \Cref{cor-linear}, $J_t(G)$ has a linear resolution if and only if $t=7,\ldots,14$.

    }
\end{example}

\section{Projective Dimension and Cohen-Macaulay Property}\label{sec pd}

In this section, we compute the projective dimension of $J_t(G)$ for a chordal graph $G$ in terms of the big height of the corresponding ideal. As a corollary, we combinatorially classify when such an ideal is Cohen-Macaulay.

\begin{proposition}\label{lem bght disjoint}
    Let $G$ be a disjoint union of two graphs $G_1$ and $G_2$, i.e., $G=G_1\sqcup G_2$. Then 
$$\b(J_{t}(G))=\b(J_{t}(G_1))+\b(J_{t}(G_2)).$$
\end{proposition}
\begin{proof}
    Since $G=G_1\sqcup G_2$, we have $\G(J_{t}(G))=\G(J_{t}(G_1))\sqcup \G(J_{t}(G_2))$. Thus, $\mathfrak{p}$ is a minimal prime ideal of $J_t(G)$ if and only if $\p=\p_{1}+\p_{2}$, where $\p_{1}$ and $\p_{2}$ are minimal prime ideals of $J_{t}(G_1)$ and $J_{t}(G_2)$, respectively. Hence, the result follows.
\end{proof}

\begin{proposition}\label{pd induced subgraph}
    For any induced subgraph $H$ of a graph $G$, $\b(J_t(H))\leq \b(J_t(G))$.
\end{proposition}
\begin{proof}
    Since $H$ is a subgraph of $G$, any connected set of cardinality $t$ in $H$ is also a connected set in $G$. Thus, we have $J_t(H)\subseteq J_t(G)$. Let $\p$ be an associated prime of $J_t(H)$ such that $\h(\p)=\b(J_t(H))$. Now, let us consider the prime ideal $\mathfrak{q}=\p+\l  V(G)\setminus V(H)\r$. Note that since $H$ is an induced subgraph of $G$, for each $m\in \G(J_t(G))\setminus \G(J_t(H))$, there exists $z\in V(G)\setminus V(H)$ such that $z\mid m$. Thus, $\mathfrak{q}$ is a prime ideal containing $J_t(G)$. Hence, there exists a minimal prime ideal $\p'$ of $J_t(G)$ such that $\p'\subseteq \mathfrak{q}$. Our aim now is to show that $\p\subseteq \p'$. Indeed, if $x\in\G(\p)$, then there exists some $m\in \G(J_t(H))$ such that $x\mid m$ and for each $x'\in\G(\p)$ with $x'\neq x$, we have $x'\nmid m$. Note that for each $y\in \l  V(G)\setminus V(H)\r$, $y\nmid m$. Thus, $m\in J_t(H)\subseteq J_t(G)$ and $\p'\subseteq \q$ implies $x\in \p'$. Hence, $\p\subseteq \p'$ and consequently, $\b(J_t(H))\leq \h(\p')\leq \b(J_t(G))$.
\end{proof}

\begin{proposition}\label{pd J}
    Let $C\subseteq V(G)$ be such that $|C|=t-1$ with $t\geq 2$ and $G[C]$ is connected. Let $J$ denote the ideal $\l w\x_{C}\mid w\in N_G(C)\r$. Then $\pd(R/J)\leq \b(J_t(G))$.
\end{proposition}
\begin{proof}
    It is easy to see that $\pd(R/J)=\pd(R/\l N_G(C)\r)$, and the Koszul complex is the minimal free resolution of $R/\l N_G(C)\r$. Thus, $\pd(R/J)=|N_G(C)|$. In order to show $\b(J_t(G))\geq |N_G(C)|$, note that $\p=\l V(G)\setminus C\r$ is a prime ideal containing $J_t(G)$. Then there exists a minimal prime ideal $\q$ of $J_t(G)$ such that $\q\subseteq \p$. Since $J\subseteq J_t(G)\subseteq \q$ and $C\cap \G(\q)=\emptyset$, we must have $\l N_G(C)\r\subseteq \q$. Hence, $\b(J_t(G))\geq \h(\q)\geq |N_G(C)|$ as desired.
\end{proof}

The following lemma on the big height of $J_t(G)$ is important in the proof of the main theorem of this section.

\begin{lemma}\label{bight lower bound}
    Let $x$ be a simplicial vertex of a graph $G$ and $C\subseteq V(G)$ be such that $x\in C$, $\vert C\vert =t-1$ for some integer $t\ge 2$, and $G[C]$ is connected. Then for each $y\in N_G(C)$, we have
    $$\b(J_t(G))\geq \vert N_G(C)\vert +\vert N_{G}(y)\setminus N_G[C]\vert +\b(J_t(G\setminus (N_G[C]\cup N_G[y])).$$
\end{lemma}
   \begin{proof}
	We divide the proof into the following two cases.

    \noindent
    {\bf Case-I.} $N_G(y)\subseteq N_G[C]$. In this case, $|N_G(y)\setminus N_G[C]|=0$. Let $\p$ be a minimal prime ideal of $J_t(G\setminus (N_G[C]\cup N_G[y]))$ such that $\h(\p)=\b(J_t(G\setminus (N_G[C]\cup N_G[y])))$. We claim that $\mathfrak{q}=\p+\l N_G(C)\r$ is a prime ideal containing $J_t(G)$. Indeed, if $f\in\G(J_t(G))$ such that $f\in J_t(G\setminus (N_G[C]\cup N_G[y]))$, then $f\in \p\subseteq \q$. Otherwise, $f\in \l N_G(C)\r$ since $|C|=t-1$, $y\in N_G(C)$, and $N_G(y)\subseteq N_G[C]$. Thus we have $f\in\q$, and consequently, $J_t(G)\subseteq \q$, as required. Next, we want to show that $\q$ is a minimal prime ideal containing $J_t(G)$. Suppose $\mathfrak{q}'$ is a minimal prime ideal of $J_t(G)$ such that $\mathfrak{q}'\subseteq \mathfrak{q}$. Observe that since $\p$ is a minimal prime ideal of $J_t(G\setminus (N_G[C]\cup N_G[y]))$, for each $x\in\G(\p)$ there exists some $g_x\in\G(J_t(G\setminus (N_G[C]\cup N_G[y])))$ such that $x\mid g_x$, and $g_x\notin \p'$, where $\G(\p')=\G(\p)\setminus\{x\}$. By definition, $g_x\in\G(J_t(G))$ and since $\G(\p)\cap N_G(C)=\emptyset$ we must have $x\in\q'$. Thus, $\p\subseteq \q'$. Also, note that $\l N_G(C)\r\subseteq \q'$ since $z\x_C\in J_t(G)$ for all $z\in N_G(C)$ and $(\{z\}\cup C)\cap \G(\p)=\emptyset$. Thus, we have $\mathfrak{q}'=\mathfrak{q}$, and consequently,
	\begin{align*}
		\b(J_t(G))\geq \h(\mathfrak{q})&=\vert N_G(C)\vert+\h(\p)\\
		&=\vert N_G(C)\vert +\vert N_{G}(y)\setminus N_G[C]\vert +\b(J_t(G\setminus (N_G[C]\cup N_G[y])),
	\end{align*}
	where $|N_G(y)\setminus N_G[C]|=0$, in this case.

    \noindent
	 {\bf Case-II.} $N_G(y)\not\subseteq N_G[C]$. Our aim in this case is to show that if $\p$ is a minimal prime ideal of $J_t(G\setminus (N_G[C]\cup N_G[y]))$ such that $\h(\p)=\b(J_t(G\setminus (N_G[C]\cup N_G[y])))$, then 
	\[
	\q_1=\p+\l x\r+\l N_G(C)\setminus\{y\}\r+\l N_G(y)\setminus N_G[C]\r
	\]
	is a minimal prime ideal containing $J_t(G)$. As before, if $f\in\G(J_t(G))$ such that we have $f\in J_t(G\setminus (N_G[C]\cup N_G[y]))$, then $f\in \p\subseteq \q_1$. Otherwise, $f\in \l N_G[C]\cup N_G[y]\r$. Observe that if $f\notin \l (N_G(C)\setminus\{y\})\cup (N_G(y)\setminus N_G[C])\r$, then $f=y\cdot\x_C$ since $|C|=t-1$. Therefore, $f\in \q_1$ since $x\in C$. Thus, $\q_1$ is a prime ideal containing $J_t(G)$. Now, to show $\q_1$ is minimal, we consider some $\q_1'\subseteq\q_1$ such that $\q_1'$ is a minimal prime ideal of $J_t(G)$, and we proceed to show that $\q_1'=\q_1$. Proceeding as in the previous case, it is easy to see that $\p\subseteq \q_1'$ since $\G(\p)\cap ((N_G(C)\setminus\{y\})\cup (N_G(y)\setminus N_G[C])\cup\{x\})=\emptyset$. Now, for each $y'\in N_G(C)$ with $y'\neq y$, we have $yy'\x_{C\setminus\{x\}}\in J_t(G)$ by \Cref{simplicial connected}. Hence, $\l N_G(C)\setminus \{y\}\r\subseteq \q_1'$ since $((C\setminus\{x\})\cup \{y\})\cap \G(\q_1)=\emptyset$. On the other hand, we have $x\in \q_1'$ since $y\x_C\in J_t(G)$. Observe that for each $w\in N_G(y)\setminus N_G[C]$, $G[C\cup\{y,w\}]$ is connected, and thus, $G[C\setminus \{x\}\cup \{y, w\}]$ is connected by \Cref{simplicial connected}. In particular, $yw\x_{C\setminus\{x\}}\in J_t(G)$ for each $w\in N_G(y)\setminus N_G[C]$, which implies $\l N_G(y)\setminus N_G[C]\r\subseteq \q_1'$ since $((C\setminus\{x\})\cup \{y\})\cap \G(\q_1)=\emptyset$. Hence, $\q_1'=\q_1$, and consequently, 
	\begin{align*}
		\b(J_t(G))&\geq \h(\mathfrak{q}_1)\\
		&=\vert N_G(C)\vert +\vert N_{G}(y)\setminus N_G[C]\vert +\b(J_t(G\setminus (N_G[C]\cup N_G[y])).
	\end{align*}
	This completes the proof. 
\end{proof}

We now proceed to prove the main theorem of this section.

\begin{theorem}\label{pd_main}
    Let $G$ be a chordal graph. Then for all $t\ge 2$, $\pd(R/J_t(G))=\b(J_t(G))$.
\end{theorem}

\begin{proof}
    Due to \Cref{pd and bight}, it is enough to show that $\pd(R/J_t(G))\leq \b(J_t(G))$. We proceed by induction on $|V(G)|$. The statement is trivially true for $|V(G)|< t$. If $|V(G)|=t$, then either $J_t(G)=\l 0\r$ or $J_t(G)=\l \x_{V(G)}\r$ depending on whether $G$ is connected or not. In both cases, it is easy to see that $\pd(R/J_t(G))=\b(J_t(G))$. Therefore, we may assume that $|V(G)|>t$. Note that, in case $G$ is disconnected, and $G_1,\ldots,G_r$ are all connected components of $G$ such that $|V(G_i)|< t$ for all $i\in[r]$, then again $J_t(G)=\l 0\r$, and thus $\pd(R/J_t(G))=\b(J_t(G))$. Therefore, we may further assume that $G$ has at least one connected component with at least $t$ many vertices. Now, let us consider the following two cases:

    \noindent
    \textbf{Case-I.} For each simplicial vertex $x$ of $G$, if $C\in \A_x$, then $V(G)\neq N_G[C]$. Note that since $G$ has at least one connected component with at least $t$ many vertices, we can find a simplicial vertex $x$ of $G$ and some $C\in \A_x$ such that $C\subsetneq N_G[C]$. Based on this observation, we consider two subcases:

    \noindent
    \textbf{Subcase-I(A).} There exists a simplicial vertex $x$ of $G$ and $C\in \A_x$ such that $G[N_G[C]]$ forms a connected component of $G$. Then we have 
\[
J_t(G)=J_t(G[N_G[C]])+J_t(G\setminus N_G[C]).
\]
    From our assumption in Case-I, it follows that 
    $$1\le |N_G[C]|<|V(G)| \text{ and } 1\le |V(G\setminus N_G[C])|<|V(G)|.$$ 
    Therefore, by the induction hypothesis, \Cref{reg and pd sum} and \Cref{lem bght disjoint}, we have 
    \begin{align*} \pd(R/J_t(G))&=\pd(R/J_t(G[N_G[C]]))+\pd(R/J_t(G\setminus N_G[C]))\\
        &\leq \b(J_t(G[N_G[C]]))+\b(J_t(G\setminus N_G[C]))\\
        &=\b(J_t(G)).
    \end{align*}
    \textbf{Subcase-I(B).} For each simplicial vertex $x$ of $G$ and each $C\in \A_x$, $G[N_G[C]]$ does not form a connected component of $G$. In particular, $C\subsetneq N_G[C]$ for any such $C$. Recall that $J_t(G)=I(\H(G,t))$. Now, fix a simplicial vertex $x$ of $G$. Let $\A_x=\{C_1,\ldots,C_k\}$. For $1\leq i\leq k$, we define
    \begin{align*}
         J_{i}= \l\x_{C_i} w\mid w\in \B_{C_i}\r\text{ and }
        K_i =I(\H(G,t)\setminus \{C_1,\ldots,C_i\}).
    \end{align*}
    Note that by construction, $\B_{C_1}\neq \emptyset$ as $C_1\subsetneq N_G[C_1]$. Fix some $i\in\{1,\ldots,k\}$ so that $\B_{C_i}\neq \phi$. Then we are in the situation of \Cref{main lemma}, and thus, $J_i+K_i=I(\H(G,t)\setminus \{C_1,\ldots,C_{i-1}\})$. Now, we proceed to prove the following claim:
    \medskip
    
    \noindent\textbf{Claim 1:} $\pd(R/J_i\cap K_i)\leq \b(J_t(G))-1$.
    \medskip
    
    \noindent\textit{Proof of the Claim 1.} We have $J_i\cap K_i=\x_{C_i}L_i$, where $L_i$ is generated by the monomials $\frac{\lcm(m,m')}{\x_{C_i}}$ with $m\in J_i$ and $m'\in K_i$. Thus, it is enough to prove $\pd(R/L_i)\leq \b(J_t(G))-1$. Let $\B_{C_i}=\{w_1,\ldots,w_s\}$. Then by \Cref{main lemma}, $L_i+\l w_1,\ldots,w_s\r=\l w_1,\ldots,w_s\r$, and hence, $\pd(R/L_i+\l w_1,\cdots, w_s\r)=s$. Note that $\B_{C_i}\subseteq N_G(C_i)$. Moreover, since $G[N_G[C_i]]$ does not form a connected component of $G$, there exists some $y\in N_G(C_i)$ such that $N_G(y)\setminus N_G[C_i]\neq \emptyset$. In this case, by \Cref{bight lower bound}, $\b(J_t(G))\ge |N_G(C_i)|+|N_G(y)\setminus N_G[C_i]|\ge s+1$. Thus, 
    \begin{align}\label{comma all variable}
       \pd(R/L_i+\l w_1,\dots, w_s\r)\le \b(J_t(G))-1. 
    \end{align}
    Now by \Cref{main lemma}, 
    \[(L_i:w_j)=\l N_G(C_i)\setminus \{w_j\}\r+\l N_G(w_j)\setminus N_G[C_i]\r+J_t(G\setminus (N_G[C_i]\cup N_G[w_j]))\]
    for each $j\in [s]$. Thus, from \Cref{reg and pd sum}(ii) it follows that
    \begin{align*}
        \pd(R/(L_i:w_j))=\, & \pd(\mathbb K[N_{G}[C_i]]/\l N_G(C_i)\setminus \{w_j\}\r)\\ &+ \pd(\mathbb K[N_{G}(w_j)\setminus N_{G}[C_i]]/\l N_G(w_j)\setminus N_G[C_i]\r)\\
        & + \pd(\mathbb K[V(G)\setminus (N_G[C_i]\cup N_G[w_j])]/J_t(G\setminus (N_G[C_i]\cup N_G[w_j]))),
    \end{align*}
    where for a subset $A$ of variables, we denote the polynomial ring $\mathbb K[x_i\mid x_i\in A]$ as $\mathbb K[A]$.
    Note that if $I\subseteq R$ is an ideal generated by $k$ indeterminates, then $\pd(R/I)=k$ (this simply follows from the observation that $\pd(R/\l x_i\r)=1$, where $x_i$ is an indeterminate, and then using \Cref{reg and pd sum} (ii), repeatedly). Thus, from the above equation and induction hypothesis, we have 
    \begin{align*}
        \pd(R/(L_i:w_j))\leq \vert N_G(C_i)\vert -1+ \vert N_G(w_j)\setminus N_G[C_i]\vert+ \b (J_t(G\setminus (N_G[C_i]\cup N_G[w_j]))).
    \end{align*}
    Finally, due to \Cref{bight lower bound}, we obtain 
    \begin{align}\label{colon Li}
        \pd(R/(L_i:w_j))\le \b(J_t(G))-1,
    \end{align} for each $j\in [s]$. Observe that $((L_i+\l w_1,\ldots,w_{s-1}\r):w_s)=(L_i:w_s)$ by \Cref{colon comma exchange} and \Cref{main lemma}. Thus, using \Cref{colon Li}, we get $$\pd(R/((L_i+\l w_1,\ldots,w_{s-1}\r):w_s))\le \b(J_t(G))-1.$$
    Consequently, by \Cref{regularity and pd lemma} and the \Cref{comma all variable}, we obtain 
    $$\pd(R/L_i+\l w_1,\ldots,w_{s-1}\r \le \b(J_t(G))-1.$$
    In view of \Cref{colon comma exchange} and \ref{main lemma}, we again have $((L_i+\l w_1,\ldots,w_{s-2}\r):w_{s-1})=(L_i:w_{s-1})$. Thus, proceeding similarly as before and using \Cref{regularity and pd lemma} repeatedly, we finally obtain $\pd(R/L_i)\le \b(J_t(G))-1$, and this completes the proof of the above claim.\par 

    It is easy to observe that $J_t(G\setminus x)=I(\H(G,t)\setminus \{C_1,\ldots,C_k\})=K_k$. Since $G\setminus x$ is an induced subgraph of $G$, by \Cref{pd induced subgraph} and the induction hypothesis, we have $\pd(R/K_k)\le \b(J_t(G))$. Moreover, $\pd(R/J_k)=\pd(R/\l \B_{C_k}\r)=|\B_{C_k}|\le |N_G(C_k)|$, and proceeding as in the proof of \Cref{pd J}, we have $\pd(R/J_k)\le \b(J_t(G))$. Thus, using the Claim 1 and by \Cref{regularity and pd lemma intersection}, we obtain 
    \[
    \pd(R/J_k+K_k)\le \b(J_t(G)),
    \]
    where the ideal $J_k+K_k=I(\H(G,t)\setminus \{C_1,\ldots,C_{k-1}\})=K_{k-1}$, by \Cref{main lemma}. Next, using \Cref{main lemma}, we write $K_{k-2}=J_{k-1}+K_{k-1}$ and continue the above process. Note that $\B_{C_1}\neq \emptyset$, and if for some $i\in \{2,\ldots,k\}$, $\B_{C_i}=\emptyset$, then $K_i=K_{i-1}$. Thus using \Cref{main lemma} and \Cref{regularity and pd lemma intersection} repeatedly, we get $\pd(R/J_i+K_i)\le \b(J_t(G))$ for each $i\in [k]$. In particular, $\pd(R/J_t(G))=\pd(R/J_1+K_1)\le \b(J_t(G))$, as desired.

     \medskip
     \noindent
    \textbf{Case-II.} There exists a simplicial vertex $x$ of $G$, and some $C_1\in \A_x$ so that $V(G)=N_G[C_1]$. In particular, $G$ is a connected graph. As before, since $G$ has at least one connected component with at least $t$ many vertices, we may as well assume that $C_1\subsetneq N_G[C_1]$. Without loss of generality, let $\A_x=\{C_1,\ldots,C_k\}$, where $N_G[C_i]=V(G)$ for $1\le i\le l$, and $N_G[C_i]\subsetneq V(G)$ for $l+1\le i\le k$. As before, define 
    \[
    J_{i}= \l\x_{C_i} w\mid w\in \B_{C_i}\r\text{ and } K_i =I(\H(G,t)\setminus \{C_1,\ldots,C_i\})
    \]
    for $i\in [k]$. Fix some $i\in [k]$ such that $\B_{C_i}\neq \emptyset$. First, we consider the case when $l+1\le i\le k$. In this case, since $G$ is connected and $N_G[C_i]\subsetneq V(G)$, there exists some $y\in N_G(C_i)$ such that $N_G(y)\setminus N_G[C_i]\neq\emptyset$. In particular, $G[N_G[C_i]]$ does not form a connected component of $G$. Thus proceeding as in Subcase-I(B), we obtain 
    $$\pd(R/J_i\cap K_i)\le \b(J_t(G))-1.$$
    Now consider the case when $i\in [l]$. Let us take
    $$\B_{C_i}=\{w_1,\ldots,w_r\} \text{ and } N_G(C_i)=\{w_1,\ldots,w_r,w_{r+1},\ldots,w_s\}$$ for some $s\ge r$. Then $J_i=\x_{C_i}\l w_1,\ldots,w_r\r$.

        \medskip
    
    \noindent\textbf{Claim 2:} For each $i\in [l]$, if $\B_{C_i}\neq\emptyset$, then $\pd(R/J_i\cap K_i)\le\b(J_t(G))-1$.
    \medskip

        \noindent\textit{Proof of the Claim 2.} We first aim to show that  
        \begin{align}\label{intersection expression}
           J_i\cap K_i=\x_{C_i}\l w_mw_n\mid 1\le m<n\le r\r+\x_{C_i}\l w_mw_n\mid m\in [r],r+1\le n\le s\r. 
        \end{align}
        Indeed, by \Cref{simplicial connected}, for each $1\le m<n\le r$, $G[(C_i\setminus \{x\})\cup\{w_m,w_n\}]$ is a connected subgraph of $G$. Thus $\x_{C_i}w_mw_n=\lcm(w_m\x_{C_i},w_mw_n\x_{C_i\setminus\{x\})}$, where $w_m\x_{C_i}\in J_i$, and $w_mw_n\x_{C_i\setminus\{x\}}\in K_i$. Similarly, if $m\in[r]$ and $r+1\le n\le s$, then by \Cref{simplicial connected}, $G[(C_i\setminus \{x\})\cup\{w_m,w_n\}]$ is also connected. Thus, if $M$ denotes the right-hand side of \Cref{intersection expression}, then $M\subseteq J_i\cap K_i$. Conversely, if $A\subseteq V(G)$ such that $C_i\nsubseteq A$, $|A|=t$, and $G[A]$ is connected, then $\{w_m,w_n\}\subseteq A$ for some $w_m,w_n\in N_G(C_i)$. Thus $J_i\cap K_i\subseteq M$ and consequently, 
        \[\pd(R/J_i\cap K_i)=\pd(R/I(G')),\] where $I(G')$ is the edge ideal of the graph $G'$ with
        \begin{align*}
          V(G')&=N_G(C_i),\\  
          E(G')&=\{\{w_m,w_n\},\{w_p,w_q\}\mid 1\le m<n\le r,p\in[r],r+1\le q\le s\}.
        \end{align*}
         It is easy to observe that the complement of $G'$ is a disconnected graph, and thus, by using \cite[Theorem 4.2.6]{JacquesThesis} we have $\pd(R/J_i\cap K_i)=|N_G(C_i)|-1$. This completes the proof of Claim 2 since $\b(J_t(G))\ge |N_G(C_i)|$, by \Cref{pd J}.

        Thus, for each $i\in [k]$ we observe that if $\B_{C_i}\neq \emptyset$, then $\pd(R/J_i\cap K_i)\le \b(J_t(G))-1$. Hence, we are in the same situation as in Subcase-I(B). Proceeding as before, we see that $\pd(R/J_i+K_i)\le \b(J_t(G))$ for each $i\in [k]$, and in particular, 
        $$\pd(R/J_t(G))=\pd(R/J_1+K_1)\le \b(J_t(G)).$$ 
        This completes the proof of the Theorem.
\end{proof}

    \begin{remark}\normalfont\label{pd char}
        Note that $\b(J_t(G))$ is the same as the maximum cardinality of a minimal vertex cover of $\mathcal{H}(G,t)$. Thus, by \Cref{pd_main}, since the projective dimension of $t$-connected ideals of chordal graphs can be expressed in terms of some combinatorial invariant of the graph, it follows that in this case the projective dimension is independent of the characteristic of the base field.
    \end{remark}

\begin{example}{\rm
Let $G$ be the graph as in \Cref{chordal G}. Note that $\{x_4,x_5,x_6,x_7,x_8,x_{12},x_{13},x_{14}\},$ is a vertex cover of $\H(G,4)$ with maximum possible cardinality. Thus by \Cref{pd_main}, $\pd(R/I_4(G))=8$.    
}
\end{example}

\begin{remark}{\rm
    In this context, one should note that if $I(\H)$ is a sequentially Cohen-Macaulay edge ideal of a hypergraph $\H$, then $\pd(R/I(\H))=\b(I(\H))$ \cite[Corollary 3.33]{MoreyVillarreal}. Also, it is well-known that if $G$ is chordal, then $I(G)$ is sequentially Cohen-Macaulay. Now, due to these facts and \Cref{pd_main}, one can ask whether, for a chordal graph $G$ and $t\geq 3$, $J_t(G)$ is sequentially Cohen-Macaulay or not. However, the answer to this question is negative (see \cite[Proposition 4.3]{ADGRS}). Moreover, a natural question that arises from this discussion is the following:
    \begin{question}
     If $I(G)$ is sequentially Cohen-Macaulay, then for all $t\ge 3$ do we have $\pd(R/J_t(G))= \b(J_t(G))$ ?   
    \end{question}
     Although this happens in the case of chordal graphs, this question has a negative answer in general. For example, if we consider the cycle $C_5$ of length $5$, then $I(C_5)$ is sequentially Cohen-Macaulay but $\pd(R/I_3(C_5))=3>2=\b(I_3(C_5))$.
    }
\end{remark}

Next, as a corollary of \Cref{pd_main}, we partially generalize a celebrated result of Herzog–Hibi–Zheng \cite{HerzogHZ2006}, which establishes the equivalence between the Cohen–Macaulay and the unmixed properties for edge ideals of chordal graphs.

\begin{corollary}\label{cor-CM}
        Let $G$ be a chordal graph and $t\ge 2$ be an integer. Then $J_{t}(G)$ is Cohen-Macaulay if and only if $J_{t}(G)$ is unmixed.
\end{corollary}
\begin{proof}
    The proof follows from \Cref{pd_main} and the Auslander-Buchsbaum formula.
\end{proof}

Note that in \cite{HerzogHZ2006}, the authors provided an explicit combinatorial characterization of Cohen--Macaulay chordal graphs and, based on this, determined the Gorenstein chordal graphs, which turn out to be precisely disjoint unions of edges. 

However, an important observation regarding the $t$-connected ideals of graphs is that two non-isomorphic graphs can have the same $t$-connected ideal for $t \ge 3$. For example, let $G_1 = K_n$ for some $n \ge 3$, and let $G_2$ be obtained from $K_n$ by removing exactly one edge (see, for instance, \Cref{completeoneedge}). Then one can observe that $J_t(G_1) = J_t(G_2)$ for each $t \ge 3$, while $G_1$ and $G_2$ are non-isomorphic. 

\begin{figure}[h!]
\centering
\begin{tikzpicture}[scale=.75,
    vertex/.style={circle, fill=black, inner sep=1.8pt},
    textlabel/.style={draw=none, fill=none, font=\small}
]

\coordinate (x1) at (90:1.5);
\coordinate (x2) at (30:1.5);
\coordinate (x3) at (330:1.5);
\coordinate (x4) at (270:1.5);
\coordinate (x5) at (210:1.5);
\coordinate (x6) at (150:1.5);

\foreach \i in {1,...,6}
  \foreach \j in {\i,...,6} {
    \ifnum\i=\j\else
      \draw (x\i) -- (x\j);
    \fi
  }

\node[vertex,label=above:$x_1$] at (x1) {};
\node[vertex,label=right:$x_2$] at (x2) {};
\node[vertex,label=right:$x_3$] at (x3) {};
\node[vertex,label=below:$x_4$] at (x4) {};
\node[vertex,label=left:$x_5$] at (x5) {};
\node[vertex,label=left:$x_6$] at (x6) {};

\node[textlabel] at (0,-2.8) {$K_6$};

\begin{scope}[shift={(5,0)}]

\coordinate (y1) at (90:1.5);
\coordinate (y2) at (30:1.5);
\coordinate (y3) at (330:1.5);
\coordinate (y4) at (270:1.5);
\coordinate (y5) at (210:1.5);
\coordinate (y6) at (150:1.5);

\foreach \i in {1,...,6}
  \foreach \j in {\i,...,6} {
    \ifnum\i=\j\else
      \ifnum\i=1 \ifnum\j=2
      \else
        \draw (y\i) -- (y\j);
      \fi
      \else
        \ifnum\i=2 \ifnum\j=1
        \else
          \draw (y\i) -- (y\j);
        \fi
      \else
        \draw (y\i) -- (y\j);
      \fi
      \fi
    \fi
  }

\node[vertex,label=above:$x_1$] at (y1) {};
\node[vertex,label=right:$x_2$] at (y2) {};
\node[vertex,label=right:$x_3$] at (y3) {};
\node[vertex,label=below:$x_4$] at (y4) {};
\node[vertex,label=left:$x_5$] at (y5) {};
\node[vertex,label=left:$x_6$] at (y6) {};

\node[textlabel] at (0,-2.8) {$K_6 \setminus \{x_1,x_2\}$};

\end{scope}

\end{tikzpicture}

\caption{$K_6$ and an edge removed.}\label{completeoneedge}
\end{figure}

Interestingly, there is a one-to-one correspondence between finite simple graphs and their edge ideals. In view of this, it is challenging to describe the combinatorial structure of chordal graphs whose $t$-connected ideal is Cohen--Macaulay for a fixed $t \ge 3$. Furthermore, it would be interesting to investigate whether the complete intersection property and the Gorenstein property coincide for $t$-connected ideals of chordal graphs, analogous to the case of edge ideals of chordal graphs.

\section{Concluding Remarks}\label{sec remark}

In this section, by \textit{hypergraphs $\H(G,t)$ induced from a graph} $G$, we mean a class of $t$-uniform hypergraphs for which $\H(G,2)=G$. In other words, the edge ideal $I(\H(G,t))$ can be viewed as a higher degree generalization of $I(G)$. In this article, we have shown in \Cref{reg_main} and \ref{pd_main} that if $I(\H(G,t))$ corresponds the ideal $J_t(G)$ of a chordal graph $G$, then $\reg(R/I(\H(G,t)))=(t-1)\nu(\H(G,t))$ and $\pd(R/I(\H(G,t)))=\b(I(\H(G,t)))$. Also, it follows from \cite[Theorem 3.12]{DRSV23} that if the complement of $G$ is chordal, then $I(\H(G,t))$ has a linear resolution. In view of this, the following question arises naturally, which nicely extends the edge ideals to a higher degree from the perspective of chordal graphs.

\begin{question}
    What type of $t$-uniform hypergraphs $\H(G,t)$ induced from a graph $G$ satisfy the following three conditions simultaneously for all $t\geq 2$:
    \begin{enumerate}
        \item[(i)] $\reg(R/I(\H(G,t)))=(t-1)\nu(\H(G,t))$ when $G$ is chordal,
        \item[(ii)] $\pd(R/I(\H(G,t)))=\b(I(\H(G,t)))$ when $G$ is chordal,
        \item[(iii)]  $I(\H(G,t))$ has a linear resolution when the complement of $G$ is chordal.
    \end{enumerate}
\end{question}

First, one may think of answering the above question for the existing classes of edge ideals of $t$-uniform hypergraphs induced from a graph, such as the $t$-path ideals and the $t$-clique ideals of graphs. Note that the $t$-path ideals fail to satisfy conditions (i) and (ii) of the above question (see \cite[Theorem 5.3, 5.8]{2024pathideals}). However, to the best of our knowledge, it is still not known whether $t$-path ideals satisfy condition (iii) or not. 

Next, let us consider the $t$-clique ideal of a graph \cite[Definition 3.1]{Moradi2018}. Then condition (iii) of the above question holds true \cite[Corollary 3.4]{Moradi2018}. We do not know about the condition (ii). However, condition (i) is not true for $ t$-clique ideals, which follows from the following example.

\begin{example}\label{clique example}{\rm
    Let $G_{t,r}=G_1\cup \cdots \cup G_{r+1}$ be a graph with $G_{i}\simeq K_t$ for all $i\in\{1,\ldots,r+1\}$ and there is a vertex $x\in V(G_{t,r})$ such that $V(G_i)\cap V(G_j)=\{x\}$ for all distinct $i$ and $j$. Let $\H(G_{t,r},t)$ be the corresponding hypergraph of the $t$-clique ideal of $G_{t,r}$. Then we have 
    $$I(\H(G_{t,r},t))=x\l \mathbf{x}_{V(G_i)\setminus\{x\}}\mid 1\leq i\leq r+1\r.$$ 
    In this case, one can observe that $\reg(R/I(\H(G_{t,r},t)))=(t-2)(r+1)+1$, whereas the induced matching number of $\H(G_{t,r},t)$ is $1$. Therefore, we have 
    $$\reg(R/I(\H(G_{t,r},t)))-(t-1)\nu(\H(G_{t,r},t))=(t-2)r.$$ 
    In other words, the regularity can be arbitrarily larger than the general lower bound for any given $t>2$.
    }
\end{example}

Moving on, in the case of edge ideals of graphs, there are several classes of graphs other than the chordal one for which the equalities $\reg(R/I(G))=\nu(G)$ (see \cite[Theorem 14]{BBH2019}) and $\pd(R/I(G))=\b(I(G))$ (for instance, sequentially Cohen-Macaulay edge ideals, etc.) hold. In this article, we have extended the above formulas in \Cref{reg_main} and \ref{pd_main} for the ideal $J_t(G)$ of a chordal graph $G$. Thus, the following question naturally arises in this context.

\begin{question}
    Find those classes of graph $G$ for which $\reg(R/J_{t}(G))=(t-1)\nu_{t}(G)$ and $\pd(R/J_{t}(G))=\b(J_t(G))$ for all $t\geq 2$, where $J_t(G)$ denotes the Stanley-Reisner ideal of $\mathrm{Ind}_t(G)$.
\end{question}

\begin{remark}\normalfont
    In \cite[Theorem 6.2]{AJM2024} it was shown that if $G$ is `gap-free and $t$-claw free', then $J_t(G)$ has linear quotients, and thus has linear resolution. In this case, one can see that $\nu_t(G)=1$, and thus gap-free and $t$-claw free graphs are an example of the class of graphs where the regularity formula in Question 5.3 is satisfied. 
\end{remark}

\noindent
{\bf Acknowledgements.} The authors would like to thank the anonymous referees for their careful reading and numerous valuable suggestions, which have significantly enhanced the clarity and readability of the paper. The first and the second authors are supported by Postdoctoral Fellowships at Chennai Mathematical Institute. The third author would like to thank the National Board for Higher Mathematics (India) for the financial support during his stay at the Chennai Mathematical Institute. All the authors are partially supported by a grant from the Infosys Foundation.

\subsection*{Data availability statement} Data sharing is not applicable to this article as no new data were created or analyzed in this study.

\subsection*{Conflict of interest} The authors declare that they have no known competing financial interests or personal relationships that could have appeared to influence the work reported in this paper.

\bibliographystyle{abbrv}
\bibliography{ref}

\end{document}